\documentclass[12pt]{amsart}
\usepackage[all]{xy}
\usepackage{amssymb}

\title{Gorenstein dimensions and Hirsch length of groups}
\author{Ioannis Emmanouil, Konstantinos Golfis and Wei Ren}

\newtheorem{Lemma}{Lemma}[section]
\newtheorem{Proposition}[Lemma]{Proposition}
\newtheorem{Theorem}[Lemma]{Theorem}
\newtheorem{Corollary}[Lemma]{Corollary}
\newtheorem{Remark}[Lemma]{Remark}
\newtheorem{Example}[Lemma]{Example}

\begin{document}

\begin{abstract}
We study the relation between the Gorenstein homological and the Gorenstein
cohomological dimension of groups. We prove that for every module of type
$FP_\infty$ over any ring the Gorenstein projective and the Gorenstein
flat dimensions are always equal to each other. In particular, we have an
equality $\mathrm{Gcd}_kG=\mathrm{Ghd}_kG$ for every group $G$ of type
$FP_\infty$ over a commutative coefficient ring $k$. For non-locally-finite
groups in Kropholler's class ${\scriptstyle\mathbf{LH}}\mathfrak F$, we
characterize the groups of Gorenstein (co)homological dimension one, in
terms of actions on trees with finite stabilizers. The Hirsch length of 
a virtually soluble group $G$ determines its Gorenstein (co)homological
dimension, even when $G$ has torsion: We show that 
$\mathrm{Ghd}_\mathbb{Z}G =h(G)$ and, if $G$ is countable,
$h(G)\leq\mathrm{Gcd}_\mathbb{Z}G\leq h(G)+1$. Some consequences for
elementary amenable groups of finite Hirsch length are also obtained.
Finally, we investigate the Gorenstein dimension of modules over certain
algebras of groups with torsion and discuss applications to classifying
spaces for proper actions.
\end{abstract}


\medskip

\maketitle
\tableofcontents

\section{Introduction}
\label{sec:int}

\noindent
The homological and cohomological dimensions of a group $G$ over a commutative
ring $k$ satisfy the inequality $\mathrm{hd}_k G\leq \mathrm{cd}_k G$, which may
very well be strict. Understanding when these two dimensions coincide and when
a positive gap occurs is a classical theme in the homological theory of groups.
For instance, if $G$ is a Bieri--Eckmann duality group over $k$, or more generally
a group of type $FP_\infty$ over $k$, then $\mathrm{hd}_k G = \mathrm{cd}_k G$.
For soluble groups, the difference between the two dimensions is closely related
to the Hirsch length. Bieri \cite[Theorem 7.10]{Bie} proved that, if $G$ is
torsion-free and soluble, then
$\mathrm{hd}_\mathbb{Z}G=h(G)$ and $h(G)\leq \mathrm{cd}_
 \mathbb{Z}G\leq h(G)+1$.
Thus, even in a well-structured class of groups, the homological and cohomological
dimensions may not be equal.

The presence of torsion makes the classical dimensions considerably less
effective; torsion provides a fundamental obstruction to the finiteness of
homological and cohomological dimensions. This phenomenon is particularly
visible for elementary amenable groups. Kropholler and Mart\'{\i}nez-P\'{e}rez
\cite{K-MP} showed that, for an elementary amenable group $G$, the dimension
$\mathrm{hd}_kG$ is finite if and only if $G$ has finite Hirsch length and
no $k$-torsion; then, $\mathrm{hd}_kG=h(G)$. Elementary amenable groups of
finite Hirsch length, however, may contain substantial torsion. Consequently,
ordinary homological and cohomological dimensions do not always provide for 
finite invariants, that are capable of detecting the Hirsch length throughout 
this class.

Gorenstein homological algebra provides a natural framework to overcome this
obstruction. Replacing projective and flat resolutions by their Gorenstein
counterparts gives rise to the Gorenstein cohomological dimension
$\mathrm{Gcd}_k G$ and Gorenstein homological dimension $\mathrm{Ghd}_k G$.
These invariants extend the classical dimensions, while remaining meaningful
for many groups with torsion. The Gorenstein cohomological dimension is
closely related to the Bredon cohomological dimension and hence to the existence
and dimension of finite dimensional models for the classifying space
$\underline{E}G$ for proper actions, thereby connecting the theory with
questions such as Brown's conjecture. More broadly, Gorenstein projective
modules are maximal Cohen--Macaulay modules under mild conditions and have
connections with representation theory, singularity theory and commutative
algebra; see for example \cite{Buc, Chr, EJ, Hap}. It is natural to study
the extent to which the classical relationship between $\mathrm{cd}_kG$
and $\mathrm{hd}_kG$ persists in the Gorenstein setting. In particular, one
may ask whether $\mathrm{Gcd}_kG=\mathrm{Ghd}_kG$ and, if equality fails, how
large the difference between the two dimensions can be. A closely related
question is to ask whether the Hirsch length continues to determine these
dimensions for virtually soluble and elementary amenable groups, under the
presence of torsion.

We first investigate the equality between the Gorenstein homological and
the Gorenstein cohomological dimension. A general module-theoretic result 
established in Appendix~A plays an important role here: For every ring $R$ 
and every $R$-module $M$ of type $FP_\infty$, we prove in Corollary
\ref{cor:GP-PGF-GF-dim} that $\mathrm{Gpd}_R M = \mathrm{Gfd}_R M$. Thus,
even though the relation between Gorenstein projective and Gorenstein flat
modules is somehow mysterious in general, the corresponding dimensions
coincide on modules of type $FP_\infty$.  Applied to the case of the
trivial $kG$-module $k$, this result shows that
$\mathrm{Gcd}_kG=\mathrm{Ghd}_kG$ for every group $G$ of type $FP_\infty$
over a commutative ring $k$; see Theorem \ref{thm:Ghd=Gcd-FP}. In particular, 
this equality holds for virtual duality groups

The classical Stallings--Swan theorem \cite{Stal,Swan} states that 
$\mathrm{hd}_{\mathbb Z}G=\mathrm{cd}_{\mathbb Z}G=1$ if and only if 
$G$ is a free group, i.e.\ if and only if $G$ acts freely on a tree.  
For a non-locally-finite group $G$ in Kropholler's class
${\scriptstyle\mathbf{LH}}\mathfrak F$, we obtain a Gorenstein analogue:
$\mathrm{Gcd}_{\mathbb Z}G=\mathrm{Ghd}_{\mathbb Z}G=1$ if and only if
$\dim\underline{E}G=1$, i.e.\ if and only if $G$ acts on a tree with finite
stabilizers; see Proposition~\ref{prop:LHF}.

We then pass to virtually soluble groups. The Hirsch length turns out to
have a particularly transparent relationship with Gorenstein homological
dimensions. Unlike the ordinary homological dimension, the Gorenstein
homological dimension detects the Hirsch length without any torsion-free
hypothesis. We show in Theorem \ref{thm:Ghd-Hir} that, for every virtually
soluble group $G$, we have equalities
$\mathrm{Ghd}_\mathbb{Z}G = \mathrm{Ghd}_\mathbb{Q}G=
\mathrm{hd}_\mathbb{Q}G =h(G)$.
Moreover, if $G$ is an elementary amenable group of finite Hirsch length,
then the above invariants are also equal to the Bredon homological dimension
$\underline{\mathrm{hd}}G$; see Corollary \ref{cor:Ghd-Hir}. Hence, we
may reinterpret \cite[Corollary 4]{FN} as follows: If $G$ is a countable
elementary amenable group of finite Hirsch length, whose Bredon cohomolgical
dimension is not equal to 2, then $G$ admits a model for $\underline{E}G$
whose dimension is either $\mathrm{Ghd}_\mathbb{Z}G$ or $\mathrm{Ghd}_\mathbb{Z}G + 1$
(Corollary \ref{cor:Ghd-Bredon}). Moreover, for a soluble group of cardinality
$<\aleph_\omega$, Mislin's characterization \cite[Corollary 4.5]{M} for the
finiteness of the Hirsch length and the dimension of $\underline{E}G$ is
extended to the finiteness of the integral and the rational Gorenstein
homological dimensions of $G$; see Corollary \ref{cor:soluble-finite-dim}.

The corresponding cohomological picture is more subtle. For a countable
virtually soluble group $G$, we establish in Theorem \ref{thm:Gcd-Hir}
that $h(G)\leq \mathrm{Gcd}_\mathbb{Q}G=\mathrm{cd}_\mathbb{Q}G
\leq \mathrm{Gcd}_\mathbb{Z}G  \leq h(G)+1$.
Thus, the Hirsch length controls the Gorenstein cohomological dimension,
but does not determine it precisely. This distinction reflects, in the
Gorenstein setting, the classical difference between homological and
cohomological dimension of soluble groups. It also bears on Talelli's
conjecture \cite{T14}, which predicts the equality
$\mathrm{Gcd}_\mathbb{Z}G= \mathrm{Gcd}_\mathbb{Q}G$
for elementary amenable groups. Our estimates in Corollary \ref{cor:Tal}
give uniform bounds on the possible discrepancy between the integral and
rational Gorenstein cohomological dimensions and hence provide some
information regarding this question.

For virtually soluble groups, the equality between the two Gorenstein
dimensions is closely related to classical finiteness and duality
properties. Building on results by Kropholler \cite{K86},
Mart\'{\i}nez-P\'erez and Nucinkis \cite{MP-N10}, we show that the
following conditions are equivalent for a virtually soluble group $G$:
(i) $\mathrm{Gcd}_\mathbb{Z}G = \mathrm{Ghd}_\mathbb{Z}G<\infty$ and
$G$ is virtually torsion-free, (ii) $G$ is of type $FP_\infty$ over
$\mathbb{Z}$ and (iii) $G$ is a virtual duality group over $\mathbb{Z}$;
see Proposition \ref{prop:v-dual}. In particular, homological finiteness,
duality and equality of the Gorenstein dimensions meet naturally in the
virtually soluble setting. It is worth noting the {\em virtually
torsion-free} condition in this characterization, since there are many
virtually soluble groups of type $FP_\infty$ which contain non-trivial
torsion. Recall that whenever $G$ has finite virtual cohomological dimension,
$\mathrm{Gcd}_\mathbb{Z}G = \mathrm{vcd}_\mathbb{Z}G$. Hence, the equality
of Gorenstein cohomological and homological dimensions for virtually soluble
groups is closely related to Brown's conjecture on classifying spaces for
proper actions; see Corollary \ref{cor:dimEG}.

Finally, we present several applications, with particular emphasis on
Gorenstein dimensions of modules over group rings and the geometric
interpretation of these dimensions, focusing on explicit virtually soluble
groups with torsion. Finite extensions $G=H\rtimes F$ of torsion-free
soluble groups $H$ of type $FP_\infty$ furnish a broad family for which
$\mathrm{Ghd}_\mathbb{Z}G=\mathrm{Gcd}_\mathbb{Z}G=h(G)$; see Proposition
\ref{prop:eg1}. For virtually soluble groups of type $FP_\infty$, we also
obtain uniform bounds for the Gorenstein projective and Gorenstein flat
dimensions of arbitrary $\mathbb{Z}G$-modules, as well as the inequalities
$h(G) \leq \mathrm{G.wgldim} \, \mathbb{Z}G
      \leq \mathrm{G.gldim} \, \mathbb{Z}G\leq h(G)+1$;
see Corollary \ref{cor:Gpd-Gfd}. On the geometric side, the Gorenstein
dimensions coincide with the Hirsch length and the Bredon cohomological and
geometric dimensions associated with the classifying space $\underline{E}G$
for proper actions. We illustrate these phenomena first by virtually abelian
and crystallographic groups, whose standard affine actions on Euclidean spaces
provide concrete models for proper actions. We then consider finite extensions
of the soluble Baumslag--Solitar groups $BS(1,n)$, which admit a natural affine
description in terms of transformations $x\mapsto \varepsilon n^k x+b$, with
$\varepsilon = {\pm1}$ and $b\in\mathbb{Z}[1/n]$; this realization makes the
occurrence of torsion through reflections particularly transparent. The
lamplighter groups $C_p\wr\mathbb{Z}$, which are metabelian but not of type
$FP_\infty$, show that beyond the $FP_\infty$ setting the Gorenstein homological
and cohomological dimensions may no longer coincide.

The paper is organized as follows. In Section 2 we recall basic definitions
and preliminary results that are used later in the paper. In Section 3 we 
study the equality between Gorenstein cohomological and homological dimensions 
for some classes of groups. Section 4 develops the theory for virtually soluble 
groups and elementary amenable groups, relating their Gorenstein dimensions to 
the Hirsch length and comparing the integral and rational dimensions. Section 5
presents examples and develops several applications regarding Gorenstein
dimensions and classifying spaces for proper actions. Finally, in Appendix A 
we prove the equality between the Gorenstein projective and Goresntein flat
dimensions for modules of type $FP_{\infty}$ (over any ring).


\section{Preliminaries}
\label{sec:pre}

\noindent
In this section, we record certain prerequisite notions and facts that
are used throughout the paper. They concern the basic definitions on
Gorenstein projective and flat modules and the Gorenstein cohomological
and homological dimensions of groups.

\vspace{0.1in}
\noindent
{\bf I.\ Gorenstein projective and Gorenstein flat modules.}
Let $R$ be an associative unital ring. Unless otherwise stated, 
all modules considered are left $R$-modules. An acyclic complex 
of projective $R$-modules
\[\mathbf{T} = \cdots\longrightarrow T_{1}\longrightarrow T_{0}
\longrightarrow T_{-1}\longrightarrow\cdots\]
is totally acyclic, if the complex of abelian groups
$\mathrm{Hom}_R(\mathbf{T}, P)$
remains acyclic for any projective $R$-module $P$. An $R$-module is
called Gorenstein projective if it is a syzygy of a totally acyclic
complex of projective $R$-modules \cite{EJ}. The Gorenstein projective
dimension $\mathrm{Gpd}_R M$ of an $R$-module $M$ is the length of a
shortest resolution of $M$ by Gorenstein projective modules; we let
$\mathrm{Gpd}_R M = \infty$ if there is no such resolution of finite
length.

An acyclic complex of flat left $R$-modules
\[\mathbf{F} = \cdots\longrightarrow F_{1}\longrightarrow F_{0}
\longrightarrow F_{-1}\longrightarrow\cdots\]
is called a totally acyclic complex of flat modules, provided that
for any injective right $R$-module $I$, the complex remains acyclic
after applying $I\otimes_R-$. A left $R$-module $M$ is called Gorenstein
flat, if it is a syzygy of a totally acyclic complex of flat modules
\cite{EJ}. The Gorenstein flat dimension $\mathrm{Gfd}_R M$ of $M$
is the length of a shortest resolution of $M$ by Gorenstein flat
modules; we let $\mathrm{Gfd}_R M = \infty$ if there is no such
resolution of finite length.

It is clear that projective modules are flat and hence
$\mathrm{fd}_R M\leq \mathrm{pd}_R M$ for any $R$-module $M$. The
relation between Gorenstein projective and Gorenstein flat modules
is more subtle in general; we can not compare directly $\mathrm{Gpd}_R M$
and $\mathrm{Gfd}_R M$ for an arbitrary $R$-module $M$. The projectively
coresolved Gorenstein flat modules (PGF-modules, for short) were introduced
by \v{S}aroch and \v{S}t$\!$'$\!$ov\'{i}\v{c}ek in \cite{SS}; these are
the syzygies of the acyclic complexes of projective modules that remain
acyclic when applying the functor $I\otimes_R-$ for any injective right
module $I$. It is clear that PGF-modules are Gorenstein flat. As shown
in \cite[Theorem 4.4]{SS}, PGF-modules are also Gorenstein projective.

\vspace{0.1in}
\noindent
{\bf II.\ Gorenstein cohomological and homological dimensions of groups.}
Let $G$ be a group and $k$ a commutative ring. We always view $k$ as a
$kG$-module with trivial group action. Then, the Gorenstein cohomological
dimension $\mathrm{Gcd}_k G$ of $G$ over $k$ is defined as the Gorenstein
projective dimension of the $kG$-module $k$; cf.\ \cite{BDT, ET18}. The
Gorenstein homological dimension $\mathrm{Ghd}_k G$ of $G$ is defined as
the Gorenstein flat dimension of the $kG$-module $k$; cf.\ \cite{ABHS, RY}.

Let $\mathrm{silp} \, k$ and $\mathrm{spli} \, k$ denote, respectively, the
supremum of the injective dimension of projective $k$-modules and the supremum
of the projective dimension of injective $k$-modules. These invariants were
introduced by Gedrich and Gruenberg \cite{GG}. If $\mathrm{silp} \, k < \infty$,
then the following conditions are equivalent (cf.\ \cite[Corollary 1.3]{ET25}
or \cite[Theorem 2.4]{R}):\\
\indent (i) $\mathrm{Gcd}_k G < \infty$;\\
\indent (ii) every $kG$-module has finite Gorenstein projective dimension;\\
\indent (iii) there exists a $k$-split $kG$-monomorphism $k\rightarrow A$,
with $A$ a $k$-projective module
satisfying $\mathrm{pd}_{kG}A <\infty$; in this case,
$\mathrm{pd}_{kG}A = \mathrm{Gcd}_k G$.

Analogously, it follows from \cite{ET25, RY} that whenever $\mathrm{sfli} \, k$,
the supremum of the flat dimension of injective $k$-modules, is finite, then
$\mathrm{Ghd}_k G<\infty$ if and only if every $kG$-module has finite Gorenstein
flat dimension, if and only if there exists a $k$-pure $kG$-monomorphism
$\iota: k\rightarrow A$, where $A$ is $k$-flat and has finite flat dimension;
in this case, $\mathrm{fd}_{kG}A = \mathrm{Ghd}_k G$.

For any commutative ring $k$, $\mathrm{silp} \, k\leq \mathrm{spli} \, k$ with
equality if $\mathrm{spli} \, k < \infty$; cf.\ \cite[Corollary 24]{DE}. In that
case, $k$ is a Gorenstein ring. Since any $k$-split map is $k$-pure and
$\mathrm{fd}_{kG}A\leq \mathrm{pd}_{kG}A$, the characterizations above show that
for any group $G$ we have an inequality $\mathrm{Ghd}_k G\leq \mathrm{Gcd}_k G$,
provided that $k$ is a commutative Gorenstein ring. We note that this inequality
is known to hold under a more general hypothesis on $k$, namely if
$\mathrm{sfli} \, k < \infty$; cf.\ \cite[Theorem 3.10]{KS}. In the present paper
though, we shall only need the inequality when $k= \mathbb{Z}$ or $\mathbb{Q}$.


\section{Gorenstein homological and cohomological dimensions}
\label{sec:eqls}

\noindent Let $G$ be a group and $k$ a commutative coefficient ring.
Since every projective module is flat, we always have
$\mathrm{hd}_k G\leq \mathrm{cd}_k G$. Even though equality does not
hold in general, it does hold for several important classes of groups.
For example, if $G$ is of type $FP_\infty$, then
$\mathrm{hd}_k G = \mathrm{cd}_k G$. In particular, if $G$ is a
Bieri--Eckmann duality group over $k$ of dimension $n$, with dualizing
module $D$, then the duality isomorphisms
\[
\mathrm{H}^i(G; M)\cong \mathrm{H}_{n-i}(G; D\otimes_k M)
\]
imply that $\mathrm{hd}_k G = \mathrm{cd}_k G = n$.
In dimension one, the Stallings--Swan theorem \cite{Stal, Swan} asserts
that $\mathrm{cd}_\mathbb{Z} G =1$ if and only if $G$ is a free group or,
equivalently by the Bass--Serre theory, if and only if $G$ acts freely on
a tree. It is clear that we always have $\mathrm{Gcd}_k G \leq \mathrm{cd}_k G$
(respectively, $\mathrm{Ghd}_k G \leq \mathrm{hd}_k G$) with equality if
$\mathrm{cd}_k G$ (respectively, $\mathrm{hd}_k G $) is finite. However,
$\mathrm{cd}_k G$ and $\mathrm{hd}_k G$ fail to be finite if $G$ has
non-trivial $k$-torsion elements. In this section, we intent to examine
the equality between the Gorenstein cohomological and homological dimensions
of groups.

The relation between Gorenstein projective and Gorenstein flat modules is
subtle in general. However, in Appendix~A we show that the two notions
behave particularly well for modules of type $FP_\infty$ over any ring $R$;
Corollary~\ref{cor:GP-PGF-GF-dim} implies that
\[
\mathrm{Gpd}_R M
=
\mathrm{Gfd}_R M
\]
for every $R$-module $M$ of type $FP_\infty$. Applied to the special case
of the group ring $R=kG$ and the trivial module $M=k$, this result yields
the following group-theoretic consequence. We note that, under an extra 
assumption on $k$, the equality $\mathrm{Ghd}_k G=\mathrm{Gcd}_k G$ for 
groups $G$ of type $FP_\infty$ over $k$ was obtained in 
\cite[Corollary 3.4]{Ster}, using special cases of some results presented 
in Appendix A. The following result is valid over any commutative coefficient 
ring.

\begin{Theorem}\label{thm:Ghd=Gcd-FP}
Let $k$ be a commutative ring. For any group $G$ of type $FP_\infty$ over $k$,
the equality $\mathrm{Gcd}_k G =\mathrm{Ghd}_k G$ holds.
\end{Theorem}

Recall that a group is said to have virtually a given property if some subgroup
of finite index has that property. A group $G$ is a virtual duality group of
virtual dimension $n$ if it has a subgroup $N$ of finite index, such that $N$
is a duality group of cohomological dimension $n$; in that case, any torsion-free
subgroup of finite index in $G$ is a duality group of cohomological dimension
$n$. However, the absolute cohomological dimension $\mathrm{cd}_k G$ may be
infinite. Standard examples include finite extensions of surface groups and
arithmetic lattices, such as $\mathrm{GL}(n,\mathbb{Z})$. As an immediate
consequence of Theorem \ref{thm:Ghd=Gcd-FP}, vitual duality groups being 
of type $FP_\infty$, we recover the corresponding equality for such groups.

\begin{Corollary}\label{cor:Ghd=Gcd-VD}
Let $G$ be a virtual duality group over a commutative ring $k$. Then,
the equality $\mathrm{Gcd}_k G =\mathrm{Ghd}_k G$ holds.
\end{Corollary}

Kropholler's class ${\scriptstyle\mathbf{H}}\mathfrak{F}$ is the smallest
group class containing all finite groups and all groups that act on a finite
dimensional contractible CW-complex with stabilizers already contained
in ${\scriptstyle\mathbf{H}}\mathfrak{F}$. The class
${\scriptstyle\mathbf{LH}}\mathfrak{F}$ consists of those groups whose finitely
generated subgroups lie in ${\scriptstyle\mathbf{H}}\mathfrak{F}$; we refer to
\cite{K93} for the precise definition. The following result may be viewed as a
proper and Gorenstein analogue of the Stallings--Swan theorem in dimension one.
Recall that the groups of integral cohomological dimension one are characterized
as free groups, or equivalently as groups acting freely on a tree. Under the
presence of torsion, freeness of the action is replaced by finiteness of the
stabilizers and is therefore related to the existence of a one-dimensional
model for the classifying space $\underline{E}G$ for proper $G$-actions. For
non-locally-finite ${\scriptstyle\mathbf{LH}}\mathfrak{F}$-groups, the
following result shows that this geometric condition is equivalent to having
both rational and integral Gorenstein cohomological dimension equal to one.

\begin{Proposition}\label{prop:LHF}
If $G$ is an ${\scriptstyle\mathbf{LH}}\mathfrak{F}$ group that is not locally
finite, then the following conditions are equivalent:
\begin{enumerate}
\item[(i)] $\mathrm{Gcd}_{\mathbb{Z}}G = \mathrm{Ghd}_{\mathbb{Z}}G = 1$,
\item[(ii)] $\mathrm{Gcd}_{\mathbb{Z}}G = 1$,
\item[(iii)] $\mathrm{cd}_{\mathbb{Q}}G = 1$ and
\item[(iv)] $\mathrm{dim} \, \underline{E}G = 1$, equivalently $G$ acts on a tree
with finite stabilizers.
\end{enumerate}
\end{Proposition}

\begin{proof}
The equivalence between (iii) and (iv) holds for all groups, as shown
in \cite[Theorem 1.3]{M}. Since $G$ is not locally finite, it is infinite
and hence both invariants $\mathrm{Gcd}_\mathbb{Z}G$ and
$\mathrm{cd}_\mathbb{Q}G$ are positive. It is then clear that
(iv)$\Rightarrow$(ii). Since $G$ is an ${\scriptstyle\mathbf{LH}}\mathfrak{F}$-group,
we have
$\mathrm{cd}_\mathbb{Q}G = \mathrm{Gcd}_\mathbb{Q}G \leq
 \mathrm{Gcd}_\mathbb{Z}G$
as shown in \cite[Theorem 3.2 and Theorem 3.5(1)]{T14}. We conclude that
(ii)$\Rightarrow$(iii). It only remains to prove that (ii)$\Rightarrow$(i).
Since $\mathrm{Ghd}_{\mathbb{Z}}G \leq \mathrm{Gcd}_{\mathbb{Z}}G =1$,
$\mathrm{Ghd}_{\mathbb{Z}}G$ is equal to either 0 or 1. But $G$ is not
locally finite and hence \cite[Corollary 5.3]{KS} implies that
$\mathrm{Ghd}_{\mathbb{Z}}G > 0$. It follows that
$\mathrm{Ghd}_{\mathbb{Z}}G = 1$ and the proof is complete.
\end{proof}

The following version of Serre's theorem on the Gorenstein homological
dimension of groups shows that $\mathrm{Ghd}_k$ is invariant under passage
to finite-index subgroups. The proof is omitted, as the argument is analogous
to that of \cite[Theorem~3.3]{R} for the Gorenstein cohomological dimension.
In fact, the result holds more generally for the Gorenstein flat dimension
of modules under a Frobenius extension of rings; recall that the extension
$kH\subseteq kG$ is Frobenius, if $H\leq G$ has finite index.

\begin{Proposition}\label{prop:Ser-Ghd}
Let $G$ be a group and $k$ a commutative ring. For any subgroup $H\leq G$
of finite index, we have an equality $\mathrm{Ghd}_k G = \mathrm{Ghd}_k H$.
\end{Proposition}

Let $G$ be a virtually torsion-free group; then, $G$ has a torsion-free
subgroup of finite index. It follows from Serre's theorem
\cite[Theorem VIII 3.1]{Bro} that all such subgroups have the same
cohomolgical dimension. The common value of these dimensions is the
virtual cohomological dimension $\mathrm{vcd}_\mathbb{Z}G$ of $G$. In
view of \cite[Theorem VIII 11.1]{Bro}, $G$ has finite virtual cohomological
dimension if and only if there exists a finite dimensional contractible
and proper $G$-CW complex $X$. As shown by Kropholler and Mislin
\cite[p.\ 126]{KM}, all groups of finite virtual cohomological dimension belong to
${\scriptstyle\mathbf{H}}_1\mathfrak{F}$. For a group $G$ which is not virtually
torsion-free, we set $\mathrm{vcd}_\mathbb{Z}G=\infty$. The virtual homological
dimension of a virtually torsion-free group is defined analogously. There are
also versions of these invariants over any commutative coefficient ring $k$.
Although the following facts are well-known, we include them here for convenient
reference and give a brief argument to prove them.

\begin{Proposition}\label{prop:vcd-vhd}
Let $G$ be a group and $k$ a commutative ring.
\begin{enumerate}
\item[(i)] There are inequalities $\mathrm{Gcd}_k G \leq \mathrm{vcd}_k G \leq \mathrm{cd}_k G$.
The first two invariants coincide if
$\mathrm{vcd}_k G <\infty$, and all three are equal if $\mathrm{cd}_k G <\infty$.
\item[(ii)] There are inequalities $\mathrm{Ghd}_k G \leq \mathrm{vhd}_k G \leq \mathrm{hd}_k G$.
The first two invariants coincide if
$\mathrm{vhd}_k G <\infty$, and all three are equal if $\mathrm{hd}_k G <\infty$.
\end{enumerate}
\end{Proposition}

\begin{proof}
We only prove the assertions of (ii), since those of (i) hold by a similar argument.
If $G$ is not virtually torsion-free, then
$\mathrm{Ghd}_k G \leq \mathrm{vhd}_k G = \mathrm{hd}_k G = \infty$. Assume that $G$
is virtually torsion-free and let $H \leq G$ be a torsion-free subgroup of finite
index. By Proposition \ref{prop:Ser-Ghd} we have $\mathrm{Ghd}_k G = \mathrm{Ghd}_k H$.
Then, $\mathrm{Ghd}_k G \leq \mathrm{hd}_k H = \mathrm{vhd}_k G \leq \mathrm{hd}_k G$.
If $ \mathrm{vhd}_k G = \mathrm{hd}_k H$ is finite, then
$\mathrm{Ghd}_k H= \mathrm{hd}_k H$, and the equality
$\mathrm{Ghd}_k G = \mathrm{vhd}_k G$ holds. If $\mathrm{hd}_k G <\infty$, then
$\mathrm{Ghd}_k G = \mathrm{hd}_k G$ and these are both equal to $\mathrm{vhd}_k G$.
\end{proof}


\section{Virtually soluble groups and Hirsch length}
\label{sec:vs-Hl}

\noindent Let $G$ be a soluble group. Recall that the Hirsch length
(or Hirsch number) of $G$, denoted by $h(G)$, is defined as the sum
of the torsion-free ranks of the derived factors of $G$:
\[h(G)=\sum\limits_{n\geq 0} \mathrm{dim}_\mathbb{Q}(G^{(n)}/G^{(n+1)} \otimes_\mathbb{Z} \mathbb{Q}).\]
This is a non-negative integer or $\infty$. The function $h$
is additive in the sense of group extensions.
If $G$ is a virtually soluble group containing a soluble subgroup $N$
of finite index, we can define the Hirsch length of $G$ as $h(N)$;
this is independent of the choice of such a subgroup $N$.

\begin{Lemma}\label{lem:Hir.len}
If $G$ is any virtually soluble group, then
$h(G)<\infty$ if and only if $\mathrm{Ghd}_\mathbb{Q}G<\infty$.
\end{Lemma}

\begin{proof}
For any group $G$, the inequality
$\mathrm{Ghd}_\mathbb{Q}G \leq \mathrm{hd}_\mathbb{Q} G$ is actually 
an equality if $\mathrm{hd}_\mathbb{Q} G < \infty$. In view of 
\cite[Theorem 1]{Stam}, for any virtually soluble group $G$ we have an 
equality $\mathrm{hd}_\mathbb{Q} G = h(G)$. Hence, $h(G)<\infty$ implies 
that $\mathrm{Ghd}_\mathbb{Q}G = \mathrm{hd}_\mathbb{Q} G = h(G) < \infty$.

To show the reverse implication, assume that $h(G)=\infty$. Then, there is
a factor $G^{(l-1)}/G^{(l)}$ in the sequence of derived factors of $G$
having infinite Hirsch length with $h(G^{(l)}) < \infty$. For any $n$
we can find a subgroup $H_n$ of $G$ containing $G^{(l)}$, such that
the quotient $H_n/G^{(l)}$ is free abelian of rank $n$; it follows that
$n \leq h(H_n)<\infty$. Then, as we noted in the first part of the proof, we
have $\mathrm{Ghd}_\mathbb{Q}H_n = \mathrm{hd}_\mathbb{Q}H_n = h(H_n) \geq n$.
Since $\mathrm{Ghd}_\mathbb{Q}G \geq \mathrm{Ghd}_\mathbb{Q}H_n$ (see,
for example, \cite[Proposition 4.1]{RY}) and $n$ is arbitrary, it follows
that $\mathrm{Ghd}_\mathbb{Q}G = \infty$.
\end{proof}

By a classical result of Stammbach \cite{Stam}, as further developed by
Bieri \cite[Theorem 7.10]{Bie}, we know that $\mathrm{hd}_\mathbb{Z}G=h(G)$
if $G$ is a torsion-free soluble group. We now state the following:

\begin{Theorem}\label{thm:Ghd-Hir}
Let $G$ be a virtually soluble group. Then, there are equalities
\[ \mathrm{Ghd}_\mathbb{Z}G= \mathrm{Ghd}_\mathbb{Q}G =
\mathrm{hd}_\mathbb{Q}G = h(G).\]
\end{Theorem}

\begin{proof}
For the virtually soluble group $G$, we have
$\mathrm{Ghd}_\mathbb{Q}G \leq \mathrm{hd}_\mathbb{Q} G = h(G)$.
If $h(G) <\infty$, we have
$\mathrm{Ghd}_\mathbb{Q}G = \mathrm{hd}_\mathbb{Q} G <\infty$.
Moreover, it follows from Lemma \ref{lem:Hir.len} that
$h(G) =\infty$ if and only if $\mathrm{Ghd}_\mathbb{Q}G = \infty$.
It follows that $\mathrm{Ghd}_\mathbb{Q}G = \mathrm{hd}_\mathbb{Q}G = h(G)$.
By \cite[Proposition 3.8]{RY}, the inequality
$\mathrm{Ghd}_\mathbb{Q} G\leq \mathrm{Ghd}_\mathbb{Z} G$
holds without any assumption on $G$. Hence, it suffices to prove that
$\mathrm{Ghd}_\mathbb{Z}G \leq h(G)$, an inequality that is obvious if
$h(G) = \infty$.

Assume that $h(G)$ is finite. Then, $G$ contains a normal soluble
subgroup $N$ of finite index, such that $h(N) = h(G) < \infty$. By
the structure theorem of Mal'cev \cite{Mal}, there is a series
\[N_0 \lhd\, N_1\lhd\, \cdots \lhd\, N_n \unlhd\, N_{n+1}=N, \]
where $N_0$ is locally finite, the index $[N: N_n]$ is finite and
the quotients $N_i/N_{i-1}$ are torsion-free abelian of finite 
torsion-free rank for $1\leq i\leq n$. Since both indices 
$[N: N_n]$ and $[G: N]$ are finite, the index $[G: N_n]$ is finite 
as well. It follows that $h(G) = h(N_n)$ and
$\mathrm{Ghd}_\mathbb{Z} G = \mathrm{Ghd}_\mathbb{Z} N_n$; cf.\
Proposition 3.4. We now consider the group extension
\[ 1 \longrightarrow N_0 \longrightarrow N_n \longrightarrow N_n/N_0
     \longrightarrow 1 \]
and invoke \cite[Proposition 4.4]{RY} to conclude that
\[ \mathrm{Ghd}_\mathbb{Z} G = \mathrm{Ghd}_\mathbb{Z} N_n \leq
   \mathrm{Ghd}_\mathbb{Z} N_0 + \mathrm{Ghd}_\mathbb{Z} (N_n/N_0) =
   \mathrm{Ghd}_\mathbb{Z} (N_n/N_0) . \]
Here, the last equality follows by applying \cite[Corollary 5.3]{KS}
to the locally finite group $N_0$. We also have $h(N_0)=0$ and hence
$h(G) = h(N_n) = h(N_n/N_0)$. Applying \cite[Theorem 7.10]{Bie}
to the torsion-free soluble group $N_n/N_0$, we conclude that
\[ \mathrm{hd}_\mathbb{Z}(N_n/N_0) = h(N_n/N_0) = h(G) < \infty . \]
Consequently, we have
\[\mathrm{Ghd}_\mathbb{Z}G \leq \mathrm{Ghd}_\mathbb{Z}(N_n/N_0)
=\mathrm{hd}_\mathbb{Z}(N_n/N_0)=h(G),\]
as needed.
\end{proof}

The theorem implies that the Gorenstein homological dimension of
virtually soluble groups (over $\mathbb{Z}$) exactly recovers the
Hirsch length, giving a well-behaved invariant even when the groups
have torsion and the ordinary (co)homological dimension is infinite.
In the next section, more examples are discussed.

\begin{Example}
The infinite dihedral group
$D_\infty = \langle t, r \mid r^2=1, rtr^{-1}=t^{-1}\rangle$ is virtually
cyclic with Hirsch length 1. In particular, it is virtually soluble and
virtually torsion-free. Since $D_\infty$ contains a torsion element
$r$ of order 2,  its cohomological and homological dimensions over
$\mathbb{Z}$ are infinite, i.e.
$\mathrm{cd}_\mathbb{Z}D_\infty = \mathrm{hd}_\mathbb{Z}D_\infty = \infty$.
However, it follows from the above theorem that
$\mathrm{Ghd}_\mathbb{Z}D_\infty = \mathrm{Ghd}_\mathbb{Q}D_\infty =
 \mathrm{hd}_\mathbb{Q}D_\infty = h(D_\infty) = 1$.
\end{Example}

Recall that the class of elementary amenable groups is the smallest
class of groups, which is closed under extensions and directed unions,
and contains all finite and all abelian groups. The Hirsch length
provides a natural measure of the torsion-free size of groups in this
class. The relationship between Hirsch length and homological dimension
was recently determined by Kropholler and Mart\'{\i}nez-P\'{e}rez
\cite{K-MP}: For any elementary amenable group $G$ and any commutative
ring $k$, one has  $\mathrm{hd}_kG<\infty$ if and only if $G$ has no
$k$-torsion and $h(G)<\infty$; whenever these conditions hold,
$\mathrm{hd}_kG =h(G)$. Elementary amenable groups of finite Hirsch
length have a particularly rigid structure \cite{HL}: There are
characteristic subgroups
\[T\leq N\leq H\leq G \]
such that $T$ is the largest normal locally finite subgroup of $G$,
$N/T$ is torsion-free nilpotent, $H/N$ is free abelian of finite rank
and $G/H$ is finite.

In particular, torsion occurs abundantly among elementary amenable groups,
constituting an obstruction to the finiteness of the integral homological
dimension. Gorenstein homological dimension behaves quite differently. The
following result shows that the presence of torsion is harmless in the
Gorenstein setting. We note that Flores and Nucinkis have related in
\cite{FN} the Gorenstein homological dimension of elementary amenable
groups to their Bredon homological dimension $\underline{\mathrm{hd}} \, G$.

\begin{Corollary}\label{cor:Ghd-Hir}
If $G$ is an elementary amenable group with $h(G)<\infty$, then
\[\mathrm{Ghd}_\mathbb{Z}G= \mathrm{Ghd}_\mathbb{Q}G
= \mathrm{hd}_\mathbb{Q}G = \underline{\mathrm{hd}} \, G = h(G).\]
\end{Corollary}

\begin{proof}
Since $h(G)$ is assumed to be finite, it follows from \cite[Theorem A]{K-MP}
that $h(G)= \mathrm{hd}_\mathbb{Q}G = \mathrm{Ghd}_\mathbb{Q}G$. Moreover,
it follows from \cite{HL} that there is an extension of groups
\[1\rightarrow T\rightarrow G\rightarrow G/T\rightarrow 1,\]
where $T$ is locally finite and $G/T$ is virtually
soluble. By Theorem \ref{thm:Ghd-Hir}, we have equalities $\mathrm{Ghd}_\mathbb{Z}(G/T)= h(G/T)= h(G)$.
Moreover, it follows from \cite[Propositions 3.8 and 4.4]{RY} that
\[ \mathrm{Ghd}_\mathbb{Q}G \leq \mathrm{Ghd}_\mathbb{Z}G \leq
   \mathrm{Ghd}_\mathbb{Z}T + \mathrm{Ghd}_\mathbb{Z}(G/T) =
   \mathrm{Ghd}_\mathbb{Z}(G/T) = h(G) . \]
Indeed, we have $\mathrm{Ghd}_\mathbb{Z}T = 0$, in view
of \cite[Corollary 5.3]{KS}. It follows that
$\mathrm{Ghd}_\mathbb{Z}G = h(G)$. Since the equality
$\underline{\mathrm{hd}} \, G = h(G)$ is proved in
\cite[Theorem 1]{FN}, the proof is complete.
\end{proof}

\begin{Corollary}\label{cor:soluble-finite-dim}
Let $G$ be a soluble group of cardinality $<\aleph_\omega$. Then the
following conditions are equivalent:
(i) $h(G)<\infty$; (ii) $\mathrm{Ghd}_\mathbb{Z}G<\infty$;
(iii) $\mathrm{Ghd}_\mathbb{Q}G<\infty$;
(iv) $\mathrm{hd}_\mathbb{Q}G<\infty$; (v) $\mathrm{cd}_\mathbb{Q}G<\infty$;
(vi) $\mathrm{dim}_G\underline{E}G<\infty$.
\end{Corollary}

\begin{proof}
By Theorem~\ref{thm:Ghd-Hir}, conditions (i)--(iv) are equivalent. The
equivalence of (i), (v) and (vi) follows from \cite[Corollary 4.5]{M}.
\end{proof}

For a countable elementary amenable group $G$, Flores and Nucinkis
\cite[Corollary 4]{FN} obtain a model for $\underline{E}G$ of
dimension at most $h(G)+1$. We may reinterpret this as follows.

\begin{Corollary}\label{cor:Ghd-Bredon}
Let $G$ be a countable elementary amenable group of finite Hirsch
length. If the Bredon cohomolgical dimension of $G$ is not equal
to 2, then $G$  admits a model for $\underline{E}G$ whose dimension
is either $\mathrm{Ghd}_\mathbb{Z}G$ or $\mathrm{Ghd}_\mathbb{Z}G + 1$.
\end{Corollary}

By \cite[Theorem 7.10]{Bie}, if $G$ is a countable torsion-free soluble
group, then $h(G)\leq \mathrm{cd}_\mathbb{Z}G \leq h(G)+1$. Thus, in the
classical setting, the cohomological dimension is not precisely determined
by the Hirsch length, in contrast with the homological dimension. A similar
distinction persists in the Gorenstein setting. While the Gorenstein homological
dimension of a virtually soluble group is equal to the Hirsch length, the
Gorenstein cohomological dimension is more delicate and may differ from it.
It is therefore natural to ask how much can $\mathrm{Gcd}_\mathbb{Z}G$
differ from the Hirsch length. We address this question for certain virtually
soluble groups below.

\begin{Theorem}\label{thm:Gcd-Hir}
If $G$ is a virtually soluble group of cardinality $\aleph_n$, then
\[h(G)\leq \mathrm{Gcd}_\mathbb{Q}G = \mathrm{cd}_\mathbb{Q}G
\leq \mathrm{Gcd}_\mathbb{Z}G\leq h(G)+n+1.\]
\end{Theorem}

\begin{proof}
For any group $G$, we have
$\mathrm{Ghd}_{\mathbb Q}G\leq\mathrm{Gcd}_{\mathbb Q}G$. Hence, for
a virtually soluble group $G$, Theorem \ref{thm:Ghd-Hir} implies that
$h(G) = \mathrm{Ghd}_{\mathbb Q}G \leq \mathrm{Gcd}_{\mathbb Q}G$.
Since virtually soluble groups belong to Kropholler's class
${\scriptstyle\mathbf{LH}}\mathfrak F$, it follows from \cite[Theorem 3.5(1)]{T14}
that $\mathrm{Gcd}_{\mathbb Q}G=\mathrm{cd}_{\mathbb Q}G$. We also have
$\mathrm{Gcd}_{\mathbb Q}G \leq \mathrm{Gcd}_{\mathbb Z}G$
(cf.\ \cite[Proposition 2.1]{ET18}) and hence
$h(G)\leq \mathrm{Gcd}_{\mathbb Q}G = \mathrm{cd}_{\mathbb Q}G \leq
 \mathrm{Gcd}_{\mathbb Z}G$.

It remains to establish the upper bound. There is nothing to prove if
$h(G)=\infty$, so assume that $h(G)$ is finite. Since the ring $\mathbb{Z}G$
has cardinality $\aleph_n$, Simson's theorem \cite{Sim} implies that any flat
$\mathbb{Z}G$-module has projective dimension $\leq n+1$. It follows from
\cite[Proposition 15]{DE} that any Gorenstein flat $\mathbb{Z}G$-module has
(PGF-dimension and hence) Gorenstein projective dimension $\leq n+1$. Theorem
\ref{thm:Ghd-Hir} implies that $\mathrm{Ghd}_{\mathbb Z}G=h(G)$ and hence the
$h(G)$-th syzygy $K$ of the trivial module $\mathbb{Z}$ in a
$\mathbb{Z}G$-projective resolution is Gorenstein flat. It follows that
$\mbox{Gpd}_{\mathbb{Z}G}K \leq n+1$ and hence
$\mbox{Gcd}_{\mathbb{Z}}G = \mbox{Gpd}_{\mathbb{Z}G} \mathbb{Z} \leq
 h(G) + \mbox{Gpd}_{\mathbb{Z}G}K \leq h(G)+n+1$.
\end{proof}

\begin{Remark}
Combining Theorems \ref{thm:Ghd-Hir} and \ref{thm:Gcd-Hir}, we obtain, for
every countable virtually soluble group $G$,
$h(G) = \mathrm{Ghd}_{\mathbb{Z}}G \leq \mathrm{Gcd}_{\mathbb{Z}}G \leq h(G)+1$.
This parallels the classical distinction between homological and cohomological
dimensions for soluble groups and is also closely related to Brown's conjecture
on the classifying space for proper actions. Recall that, if $G$ has finite virtual
cohomological dimension, then $\mathrm{Gcd}_\mathbb{Z}G = \mathrm{vcd}_\mathbb{Z}G$
(Proposition \ref{prop:vcd-vhd}(i)). Brown's conjecture asks whether such a group
admits a model for $\underline{E}G$ of dimension $\mathrm{vcd}_\mathbb{Z}G$.
Mart\'{\i}nez-P\'{e}rez and Nucinkis \cite[Corollary 5.2]{MP-N10} proved this
for virtually torsion-free elementary amenable groups. The above inequalities
place the Gorenstein cohomological dimension within the same one-dimensional
range determined by $h(G)$, while $\mathrm{Ghd}_\mathbb{Z}G$ always realizes
the lower endpoint.
\end{Remark}

Talelli \cite{T14} conjectured that
$\mathrm{Gcd}_\mathbb{Z}G = \mathrm{Gcd}_\mathbb{Q}G$ for every elementary
amenable group $G$. We do not establish the conjectured equality, but the
preceding result yields uniform bounds on the possible difference between
the integral and rational Gorenstein cohomological dimensions. In particular,
we obtain the following estimates for virtually soluble and, more generally,
elementary amenable groups.

\begin{Corollary}\label{cor:Tal}
Let $G$ be a countable group. Then,
$\mathrm{Gcd}_\mathbb{Z}G \leq \mathrm{Gcd}_\mathbb{Q}G + 1$ if $G$ is virtually
soluble, and $\mathrm{Gcd}_\mathbb{Z}G \leq \mathrm{Gcd}_\mathbb{Q}G + 2$ if $G$
is elementary amenable of finite Hirsch length.
\end{Corollary}

\begin{proof}
The first assertion follows directly from Theorem \ref{thm:Gcd-Hir}.
Let $G$ be an elementary amenable and countable group with $h(G)<\infty$.
Hillman and Linnell \cite{HL} have shown that $G$ has a unique maximal
locally finite normal subgroup $T$, such that $G/T$ is virtually soluble.
Of course, the group $G/T$ is countable and hence Theorem \ref{thm:Gcd-Hir}
implies that
\[ \mathrm{Gcd}_\mathbb{Z}(G/T) \leq h(G/T)+1 = h(G)+1 . \]
Applying \cite[Proposition 2.9]{ET18}, we conclude that
$\mathrm{Gcd}_\mathbb{Z}G \leq
 \mathrm{Gcd}_\mathbb{Z}(G/T) + \mathrm{Gcd}_\mathbb{Z}T$.
Since $G$ is countable, so is its locally finite subgroup $T$.
It follows from \cite[Corollary 3.4]{ET18} that
$\mathrm{Gcd}_\mathbb{Z}T\leq 1$ and hence
\[ \mathrm{Gcd}_\mathbb{Z}G \leq   \mathrm{Gcd}_\mathbb{Z}(G/T) + 1 \leq
   h(G) + 2 . \]
Using Corollary 4.4, we have
$h(G) = \mathrm{Ghd}_\mathbb{Q}G \leq \mathrm{Gcd}_\mathbb{Q}G$ and hence
$\mathrm{Gcd}_\mathbb{Z}G \leq \mathrm{Gcd}_\mathbb{Q}G + 2$.
\end{proof}

We shall conclude this Section, by stating a few equivalent conditions
for virtually soluble groups, that may be viewed as a Gorenstein complement
to Kropholler's result \cite[Theorem]{K86}, \cite[Theorem 1.1]{MP-N10}.

\begin{Proposition}\label{prop:v-dual}
If $G$ is a virtually soluble group, the following conditions are equivalent:
\begin{enumerate}
\item[(i)] $\mathrm{Gcd}_\mathbb{Z}G = \mathrm{Ghd}_\mathbb{Z}G<\infty$
and $G$ is virtually torsion-free,
\item[(ii)] $G$ is a virtual duality group over $\mathbb{Z}$ and
\item[(iii)] $G$ is of type $FP_\infty$ over $\mathbb{Z}$.
\end{enumerate}
\end{Proposition}

\begin{proof}
(i)$\Rightarrow$(ii): Let $H \leq G$ be a torsion-free soluble subgroup
of finite index; then, we have
\[ \mathrm{hd}_\mathbb{Z}H = h(H) = h(G) = \mathrm{Ghd}_{\mathbb{Z}} G . \]
In the above chain of equalities, the first one follows from
\cite[Theorem 7.10]{Bie} and the last one from Theorem \ref{thm:Ghd-Hir}.
On the other hand, \cite[Theorem 3.5(ii)]{T14} implies that
$\mathrm{cd}_\mathbb{Z}H = \mathrm{Gcd}_{\mathbb{Z}} H$ and hence
\cite[Theorem 3.3]{R} shows that
$\mathrm{cd}_\mathbb{Z}H = \mathrm{Gcd}_{\mathbb{Z}} G$. Since
$\mathrm{Gcd}_\mathbb{Z}G = \mathrm{Ghd}_\mathbb{Z}G<\infty$, it follows
that $\mathrm{cd}_\mathbb{Z}H = \mathrm{hd}_\mathbb{Z}H<\infty$. Invoking
\cite[Theorem]{K86}, we now conclude that $H$ is a duality group, so that
$G$ is a virtual duality group.

(ii)$\Rightarrow$(iii): By assumption, $G$ has a subgroup $H$ of finite
index, which is a duality group; then, $H$ is of type $FP_\infty$. Since
$H$ has finite index in $G$, it follows that $G$ is also of type $FP_\infty$.

(ii)$\Rightarrow$(i): Let $H \leq G$ be a soluble subgroup of finite index.
Then, $H$ is also of type $FP_\infty$ and hence \cite[Theorem]{K86} implies
that $\mathrm{cd}_\mathbb{Z}H = \mathrm{hd}_\mathbb{Z}H<\infty$. It follows
that $H$ is torsion-free (so that $G$ is virtually torsion-free) and
$\mathrm{Gcd}_\mathbb{Z}H = \mathrm{Ghd}_\mathbb{Z}H <\infty$. Invoking
\cite[Theorem 3.3]{R} and Proposition 3.4, we conclude that
$\mathrm{Gcd}_\mathbb{Z}G = \mathrm{Ghd}_\mathbb{Z}G <\infty$.
\end{proof}

\begin{Corollary}\label{cor:dimEG}
Let $G$ be a virtually soluble and virtually torsion-free group with
$\mathrm{Gcd}_\mathbb{Z}G = \mathrm{Ghd}_\mathbb{Z}G <\infty$. Then,
$G$ admits a finitely dominated model for $\underline{E}G$.
\end{Corollary}

\begin{proof}
In view of Proposition \ref{prop:v-dual}, the group $G$ is of type
$FP_\infty$ over $\mathbb{Z}$. The result then follows from
\cite[Theorem 4.1]{MP-N10}.
\end{proof}


\section{Examples and applications}
\label{sec:eg-ap}

\noindent In this section, we give some examples and applications, 
regarding modules over group rings and classifying spaces for proper 
actions. We provide finite invariants for groups with torsion and a 
natural framework in which the classical homological--cohomological 
gap can be studied.

\vspace{0.1in}
\noindent
{\bf I. Uniform bounds for modules over $\mathbb{Z}G$.} One advantage 
of Gorenstein dimensions is that they carry meaningful information, 
even under the presence of torsion.  If the group $G$ contains non-trivial 
torsion, then both $\mathrm{hd}_\mathbb{Z}G$ and $\mathrm{cd}_\mathbb{Z}G$ 
are infinite. The Gorenstein dimensions, however, may remain finite and 
retain homological information, even when the ordinary integral dimensions 
fail to do so.

This phenomenon is not restricted to isolated examples. The $FP_\infty$ 
property is invariant under passage to finite-index subgroups and 
extensions by finite groups, whereas the Hirsch length of a virtually 
soluble group remains unchanged when passing to a subgroup of finite index. 
Consequently, if $H$ is a torsion-free soluble group of type $FP_\infty$ 
and $F$ is a finite group acting on $H$, then $G=H\rtimes F$ is virtually 
soluble of type $FP_\infty$ and satisfies $h(G) = h(H)$. Since the extension 
is split, $F$ embeds in $G$, so $G$ contains torsion. This simple observation 
produces an abundance of examples to which our results apply. For this class 
of groups, Proposition \ref{prop:v-dual} yields that the Gorenstein dimensions 
recover the Hirsch length, even when the ordinary integral dimensions fail to 
provide finite invariants.

\begin{Proposition}\label{prop:eg1}
Let $H$ be a torsion-free soluble group of type $FP_\infty$. If
$F$ is a finite group acting on $H$ and $G = H \rtimes F$, then
$\mathrm{Ghd}_\mathbb{Z}G=\mathrm{Gcd}_\mathbb{Z}G = h(G) < \infty$.
\end{Proposition}

Although Gorenstein projective resolutions are difficult to describe
explicitly in general, the preceding equalities have consequences not
only for the trivial module but for arbitrary modules over the integral
group ring. Moreover, the Gorenstein global dimension and the Gorenstein
weak global dimension of the group ring are controlled by the Hirsch length.

\begin{Corollary}\label{cor:Gpd-Gfd}
Let $G$ be a virtually soluble group of type $FP_\infty$.
\begin{enumerate}
\item[(i)] For every $\mathbb{Z}G$-module $M$, $\mathrm{max}\{\mathrm{Gpd}_{\mathbb ZG}M, \mathrm{Gfd}_{\mathbb ZG}M\}\leq
h(G)+1.$
\item[(ii)] $h(G)\leq\mathrm{G.wgldim}\,\mathbb{Z}G
\leq\mathrm{G.gldim}\,\mathbb{Z}G\leq h(G)+1$.
\end{enumerate}
\end{Corollary}

\begin{proof}
Assertion (i) is immediate from \cite[Corollary~1.5]{ET18} and
\cite[Corollary 3.6]{RY}. Then, (ii) follows from the equality
$h(G) = \mathrm{Ghd}_\mathbb{Z}G$ of Theorem \ref{thm:Ghd-Hir},
the inequality of \cite[Proposition 5.8]{RY} and (i).
\end{proof}

\vspace{0.1in}
\noindent
{\bf II. Classifying spaces for proper actions.}
There is also a geometric interpretation of the above equalities. Recall
that Mart\'{\i}nez-P\'erez and Nucinkis \cite[Theorem 4.1]{MP-N10} proved
that a virtually soluble group $G$ of type $FP_\infty$ admits a finitely
dominated model for $\underline{E}G$ of dimension $h(G)$. Such a group is
virtually torsion-free and its virtual cohomological dimension is equal to
$h(G)$. Therefore, no model for $\underline{E}G$ can have dimension $< h(G)$.
It follows that the Bredon geometric dimension $\underline{\mathrm{gd}}\,G$
equals $h(G)$; consequently, for the Bredon cohomological dimension we also
have $\underline{\mathrm{cd}}\,G=h(G)$. Combining \cite[Theorem 4.1]{MP-N10}
with our results, we obtain the following string of equalities.

\begin{Corollary}\label{cor:gdG}
Let $G$ be a virtually soluble group of type $FP_\infty$. Then
\[\mathrm{Ghd}_\mathbb{Z}G=\mathrm{Gcd}_\mathbb{Z}G
=h(G)=\underline{\mathrm{cd}}\,G . \]
\end{Corollary}

This provides another connection with Brown's conjecture.
Mart\'{\i}nez-P\'erez and Nucinkis have actually proved in
\cite[Corollary~5.2]{MP-N10} that Brown's conjecture holds
more generally for virtually torsion-free elementary amenable
groups.

\vspace{0.1in}
\noindent
{\bf III. Virtually abelian and crystallographic groups.}
We next give some representative families in which the aforementioned
conclusions can be explicitly seen. The simplest general family is
$G=F\times\mathbb{Z}^n$, where $F$ is a non-trivial finite group and
$n\geq0$.  The subgroup $\mathbb{Z}^n$ has finite index in $G$, so $G$
is virtually free abelian and of type $FP_\infty$. Since $h(G)=n$, our
results give $\mathrm{Ghd}_\mathbb{Z}G=\mathrm{Gcd}_\mathbb{Z}G=n$,
whereas $\mathrm{hd}_\mathbb{Z}G=\mathrm{cd}_\mathbb{Z}G=\infty$.

This family already contains several familiar examples. If $n=0$ and
$F=C_m$, we obtain a finite cyclic group, for which both Gorenstein
dimensions are zero.  If $n=1$ and $F=S_3$, then
$G = S_3 \times \mathbb{Z}$ is a virtually cyclic group with elements
of orders $2$ and $3$, for which
$\mathrm{Ghd}_\mathbb{Z}G = \mathrm{Gcd}_\mathbb{Z}G=1$. Letting $F$ be
any finite group, we conclude that torsion of considerable complexity
may occur without affecting the value of the Gorenstein dimensions.

A geometric variant is obtained from crystallographic groups.  Let
$C_2=\langle\sigma\rangle$ act on $\mathbb{Z}^n$ by
$\sigma(v)=-v$, and set $\Gamma_n=\mathbb{Z}^n\rtimes C_2$.
Then, $\Gamma_n$ is virtually free abelian, contains an involution
and has Hirsch length $n$.  Therefore, we have
$\mathrm{Ghd}_\mathbb{Z}\Gamma_n
=\mathrm{Gcd}_\mathbb{Z}\Gamma_n=n$.
The standard affine action of $\Gamma_n$ on $\mathbb R^n$, in which
$\mathbb{Z}^n$ acts by translations and $\sigma$ acts by
$x\mapsto -x$, is proper and cocompact, providing thus a model for
$\underline{E}\Gamma_n$. In this case, the equality
$\underline{\mathrm{gd}}\,\Gamma_n=n$ is geometrically visible.

More generally, if $F$ is a finite subgroup of
$\mathrm{GL}_n(\mathbb{Z})$, then the split extension
$\mathbb{Z}^n\rtimes F$ is virtually free abelian of type
$FP_\infty$ and has Hirsch length $n$.  Since $F$ embeds in
the semidirect product, these groups provide a large family
of crystallographic examples with torsion and every positive
Hirsch length.

\vspace{0.1in}
\noindent
{\bf IV. Non-polycyclic metabelian groups.}
The preceding examples are all virtually abelian.  The same phenomenon
occurs for soluble groups, which are not virtually polycyclic. Let
$n\geq 2$ and consider the soluble Baumslag--Solitar group \cite{BS}
\[H_n=BS(1,n)=\langle a,t\mid tat^{-1}=a^n\rangle.\]
It has the semidirect-product description
$H_n\cong\mathbb{Z}[1/n]\rtimes\mathbb{Z}$, where the generator of the
infinite cyclic factor acts on $\mathbb{Z}[1/n]$ as multiplication by 
$n$.  In particular, $H_n$ is torsion-free and metabelian. Its ascending 
HNN decomposition shows that it is constructible and hence of type 
$FP_\infty$; cf.\ \cite[Theorem]{K86}. Its Hirsch length is $h(H_n)=2$.

There is an automorphism $\alpha$ of $H_n$ of order two defined by
$\alpha(a)=a^{-1}$ and $\alpha(t)=t$.  Form the split extension
$G_n=H_n\rtimes_\alpha C_2 =\langle a,t,s\mid tat^{-1}=a^n,\,
s^2=1,\,sas^{-1}=a^{-1},\,sts^{-1}=t\rangle$.
The subgroup $H_n$ has index two in $G_n$, while $s$ is an involution.
Thus $G_n$ is a soluble group of type $FP_\infty$ with torsion and
$h(G_n)=2$.  Consequently,
$\mathrm{Ghd}_\mathbb{Z}G_n=\mathrm{Gcd}_\mathbb{Z}G_n=2$.
Moreover, every $\mathbb{Z}G_n$-module $M$ satisfies
$\mathrm{Gpd}_{\mathbb ZG_n}M\leq3$ and
$\mathrm{Gfd}_{\mathbb ZG_n}M\leq3$.

The same example admits a concrete affine description.  The group
$G_n$ may be identified with the group of transformations
$x \mapsto \varepsilon n^k x+b$, where $\varepsilon \in \{1,-1\}$,
$k\in\mathbb{Z}$, and $b\in\mathbb{Z}[1/n]$.  The subgroup with
$\varepsilon=1$ is $BS(1,n)$, while $x\mapsto -x$ is the involution.
This realization makes transparent the passage from a torsion-free
soluble group of type $FP_\infty$ to a finite extension with torsion,
without changing the Hirsch length or the Gorenstein dimensions.

\vspace{0.1in}
\noindent
{\bf V. Groups outside the $FP_\infty$ class.}
The $FP_\infty$ hypothesis is indispensable in the discussion
above. The lamplighter groups provide a useful contrast. Let
$p$ be a prime and $L_p=C_p\wr\mathbb{Z}$; then,
$L_p=(\bigoplus_{i\in\mathbb{Z}}C_p)\rtimes\mathbb{Z}$, with
$\mathbb{Z}$ acting by shifting the factors.  The group $L_p$
is metabelian and has Hirsch length $h(L_p)=1$. It is a split
extension of an infinite abelian torsion group by $\mathbb{Z}$
and \cite[Theorem III.11]{BK} gives
$\mathrm{cd}_\mathbb{Q}L_p=h(L_p)+1=2$. Its torsion subgroup
$\bigoplus_{i\in\mathbb{Z}}C_p$ intersects non-trivially every
subgroup of finite index, so that $L_p$ is not virtually
torsion-free. Therefore, $L_p$ is not of type $FP_\infty$.
Theorems \ref{thm:Ghd-Hir} and \ref{thm:Gcd-Hir} on Gorenstein
homological and cohomological dimensions apply though,
showing that $\mathrm{Ghd}_\mathbb{Z}L_p = 1$ and
$\mathrm{Gcd}_\mathbb{Z}L_p = \mathrm{Gcd}_\mathbb{Q}L_p =
 \mathrm{cd}_\mathbb{Q}L_p = 2$.

This example highlights two things.  First, the
identity $\mathrm{Ghd}_\mathbb{Z}G=h(G)$ extends well beyond groups of
type $FP_\infty$ and remains valid in the presence of large locally finite
normal subgroups.  Second, without a finiteness condition such as
$FP_\infty$, the Gorenstein cohomological dimension need not agree with
the Gorenstein homological dimension.  The distinction between the two
Gorenstein invariants therefore reflects, in a precise way, the familiar
difference between homological and cohomological dimensions in the theory
of soluble groups.


\appendix

\section{Gorenstein modules of type $FP_{\infty}$}

\noindent
In this Appendix, we let $R$ be a unital associative ring and
consider, unless otherwise specified, only left $R$-modules.
We shall prove that an $R$-module of type $FP_{\infty}$ is
Gorenstein projective if and only if it is Gorenstein flat.
To that end, let ${\tt GProj}(R)$, ${\tt GFlat}(R)$ and
${\tt PGF}(R)$ be the classes of Gorenstein projective,
Gorenstein flat and projectively coresolved Gorenstein
flat $R$-modules, respectively. As noted in Section
\ref{sec:pre}, there are inclusions
\[ {\tt GProj}(R) \supseteq {\tt PGF}(R) \subseteq
   {\tt GFlat}(R). \]
We also denote by ${\tt GProj}_{\infty}(R)$,
${\tt GFlat}_{\infty}(R)$ and ${\tt PGF}_{\infty}(R)$
the subclasses of the classes above consisting of all
modules of type $FP_{\infty}$ therein. We shall prove
that the inclusions above induce equalities
\[ {\tt GProj}_{\infty}(R) = {\tt PGF}_{\infty}(R) =
    {\tt GFlat}_{\infty}(R) . \]

We begin with the case of Gorenstein projective modules
and record the following lemma, which is certainly well-known.

\begin{Lemma}\label{lem:f.g.GP}
Let $M$ be a finitely generated Gorenstein projective module.
\begin{enumerate}
\item[(i)]
There exists a short exact sequence of modules
\[ 0 \longrightarrow M \longrightarrow P \longrightarrow N
     \longrightarrow 0 , \]
where $P$ is finitely generated projective and $N$ is
finitely generated Gorenstein projective.
\item[(ii)]
There exists an exact sequence of modules
\[ 0 \longrightarrow M \longrightarrow P_{-1}
     \longrightarrow P_{-2} \longrightarrow \cdots , \]
where $P_{-1},P_{-2}, \ldots$ are finitely generated projective
and all kernels are finitely generated Gorenstein projective.
\end{enumerate}
\end{Lemma}

\begin{proof}
(i) The module $M$ is a kernel of a totally acyclic complex
of projective modules. By adding suitable disk complexes to
the latter, we may assume that the projective modules involved
are actually free. Hence, there is a set $I$ and a monomorphism
$\iota : M \longrightarrow R^{(I)}$, whose cokernel is Gorenstein
projective. Since $M$ is finitely generated, there is a finite
subset $I_0 \subseteq I$, such that $\iota$ factors through the
inclusion $R^{(I_0)} \hookrightarrow R^{(I)}$ by means of
$\iota_0 : M \longrightarrow R^{(I_0)}$. Since
$\mathrm{coker} \, \iota = \mathrm{coker} \, \iota_0 \oplus
 R^{(I \setminus I_0)}$,
the closure of the class ${\tt GProj}(R)$ under direct summands
(cf.\ \cite[Theorem 2.5]{H}) implies that
$\mathrm{coker} \, \iota_0 \in {\tt GProj}(R)$. Then, the short
exact sequence
\[ 0 \longrightarrow M
     \stackrel{\iota_0}{\longrightarrow} R^{(I_0)}
     \longrightarrow \mathrm{coker} \, \iota_0
     \longrightarrow 0 \]
is of the required type.

(ii) The construction of the exact sequence is performed by an
iterated application of (i).
\end{proof}

We recall that a class of modules is called definable if it is 
closed under products, direct limits and pure submodules. For 
any module $M$, we denote by  $\langle M \rangle$ the smallest
definable class that contains $M$; in other words,
$\langle M \rangle$ is the intersection of all definable classes 
containing $M$. We call $\langle M \rangle$ the definable closure 
of $M$.

\begin{Proposition}\label{prop:GP=PGF}
If $M$ is a Gorenstein projective module of type $FP_{\infty}$,
then $M$ is projectively coresolved Gorenstein flat. In other
words, there is an equality
${\tt GProj}_{\infty}(R) = {\tt PGF}_{\infty}(R)$.
\end{Proposition}

\begin{proof}
Since $M$ is of type $FP_{\infty}$, it admits a projective
resolution
\[ \cdots \longrightarrow P_1 \longrightarrow P_0
   \longrightarrow M \longrightarrow 0 \]
consisting of finitely generated projective modules.
Since $M$ is Gorenstein projective and the class ${\tt GProj}(R)$
is projectively resolving by \cite[Theorem 2.5]{H}, all kernels
of that resolution are also Gorenstein projective. Splicing that
projective resolution of $M$ with the coresolution of Lemma A.1(ii),
we obtain an acyclic complex ${\bf P}$
\[ \cdots \longrightarrow P_1 \longrightarrow P_0
   \longrightarrow P_{-1} \longrightarrow P_{-2}
   \longrightarrow \cdots \]
consisting of finitely generated projective modules with
$M = Z_{-1}{\bf P}$ and all kernels $(Z_n{\bf P})_n$ Gorenstein
projective. It follows that ${\bf P}$ is a totally acyclic complex
of projective modules. Let $\mathcal{C}$ be the class of those
modules $C$ for which the complex of abelian groups
$\mathrm{Hom}_R({\bf P},C)$ is acyclic. Then, $\mathcal{C}$ contains
all projective modules; in particular, it contains the regular
module $R$. We shall prove that $\mathcal{C}$ is a definable class.
Having proved that, it will follow that $\mathcal{C}$ contains the
definable closure $\langle R \rangle$ of $R$, so that the complex
of abelian groups $\mathrm{Hom}_R({\bf P},C)$ is acyclic for any
$C \in \langle R \rangle$. In view of \cite[Corollary 4.5]{SS},
it will then follow that $M \in {\tt PGF}(R)$.

The class $\mathcal{C}$ is closed under products, since
$\mathrm{Hom}_R({\bf P}, -)$ commutes with products and any product
of acyclic complexes is acyclic. The acyclic complex of abelian
groups $\mathrm{Hom}_R({\bf P},R)$ is actually a complex of right
$R$-modules, by means of the right action of $R$ on $R$. Moreover,
the complex ${\bf P}$ consisting of finitely generated projective
modules, the natural map
\[ \mathrm{Hom}_R({\bf P},R) \otimes_RC \longrightarrow
   \mathrm{Hom}_R({\bf P},C) \]
is an isomorphism of complexes for any $R$-module $C$. It follows
that $\mathcal{C}$ is closed under direct limits, since
$\mathrm{Hom}_R({\bf P},R) \otimes_R \_\!\_$ commutes with direct
limits and any direct limit of acyclic complexes is acyclic.
Finally, the closure of $\mathcal{C}$ under pure submodules follows
by applying the following lemma to the complex $\mathrm{Hom}_R({\bf P},R)$.
\end{proof}

\begin{Lemma}\label{lem:pure}
Let ${\bf X}$ be an acyclic complex of right $R$-modules and consider
a left $R$-module $C$. If the complex of abelian groups ${\bf X} \otimes_RC$
is acyclic, then the complex ${\bf X} \otimes_RC'$ is acyclic for any pure
submodule $C' \subseteq C$.
\end{Lemma}

\begin{proof}
Let $(Z_n)_n$ be the kernels of ${\bf X}$. Then, the short exact sequence
of right $R$-modules
\[ 0 \longrightarrow Z_n \longrightarrow X_n
     \longrightarrow Z_{n-1} \longrightarrow 0 \]
induces a short exact sequence of abelian groups
\[ 0 \longrightarrow Z_n \otimes_RC
     \longrightarrow X_n \otimes_RC
     \longrightarrow Z_{n-1} \otimes_RC
     \longrightarrow 0 \]
for all $n$. We have to show that this is also the case if $C$
is replaced by any pure submodule $C'$ of it. In other words,
we have to show that the additive map
\[ Z_n \otimes_RC' \longrightarrow X_n \otimes_RC' \]
is injective for all $n$ and all such $C'$'s. This follows
from the commutativity of the square
\[\xymatrix{
Z_n \otimes_RC'\ar[r] \ar[d] &X_n \otimes_RC'\ar[d]\\
Z_n \otimes_RC \ar[r] & X_n \otimes_RC
}\]
since the left vertical arrow is injective (in fact, both
vertical arrows are injective), in view of the purity of
the embedding $C' \hookrightarrow C$.
\end{proof}

Having proved Proposition \ref{prop:GP=PGF}, we now switch to the case of
Gorenstein flat modules. The following lemma is the analogue
of Lemma \ref{lem:f.g.GP} in this setting. Recall that, given a module
class $\mathcal{C}$, a linear map $f : M \longrightarrow N$ is
called $\mathcal{C}$-injective if the additive map
\[f^* : \mathrm{Hom}_R(N,C) \longrightarrow \mathrm{Hom}_R(M,C)\]
is surjective for all $C \in \mathcal{C}$. Whenever a composition
$g \circ h$ of two linear maps is $\mathcal{C}$-injective, the
map $h$ is also $\mathcal{C}$-injective.

\begin{Lemma}\label{lem:seqs}
Let $M$ be a finitely presented Gorenstein flat module.
\begin{enumerate}
\item[(i)]
There exists a short exact sequence of modules
\[ 0 \longrightarrow M \longrightarrow P \longrightarrow N
     \longrightarrow 0 , \]
where $P$ is finitely generated projective and $N$ is
finitely presented Gorenstein flat.
\item[(ii)]
There exists an exact sequence of modules
\[ 0 \longrightarrow M \longrightarrow P_{-1}
     \longrightarrow P_{-2} \longrightarrow \cdots , \]
where $P_{-1},P_{-2}, \ldots$ are finitely generated projective
and all kernels are finitely presented Gorenstein flat.
\end{enumerate}
\end{Lemma}

\begin{proof}
(i) The module $M$ is a kernel of a totally acyclic complex
of flat modules. Hence, there is a flat module $F$ and a
monomorphism $\iota : M \longrightarrow F$, whose cokernel
is Gorenstein flat. The linear map $\iota$ factors through
a suitable finitely generated free (and hence projective)
module $P$; see, for example, \cite[Theorem 4.32]{L}. Hence,
there are linear maps $\jmath : M \longrightarrow P$ and
$f: P \longrightarrow F$, such that $\iota = f \circ \jmath$;
of course, $\jmath$ is a monomorphism. Since $\imath$ is a
monomorphism with Gorenstein flat cokernel, it is
${\tt GFlat}(R)^{\perp}$-injective. It follows that $\jmath$
is ${\tt GFlat}(R)^{\perp}$-injective as well. Since $\jmath$
is a monomorphism and $P$ is free, we conclude that
$\mathrm{Ext}^1_R(\mathrm{coker} \, \jmath,C)=0$ for any
$C \in {\tt GFlat}(R)^{\perp}$. In view of
\cite[Corollary 4.12]{SS}, it follows that
$\mathrm{coker} \, \jmath$ is Gorenstein flat. Hence, the short
exact sequence
\[ 0 \longrightarrow M
     \stackrel{\jmath}{\longrightarrow} P
     \longrightarrow \mathrm{coker} \, \jmath
     \longrightarrow 0 \]
is of the required type.

(ii) The construction of the exact sequence is performed by an
iterated application of (i).
\end{proof}

\begin{Proposition}\label{prop:GF=PGF}
If $M$ is a Gorenstein flat module of type $FP_{\infty}$, then
$M$ is projectively coresolved Gorenstein flat. In other words,
there is an equality ${\tt GFlat}_{\infty}(R) = {\tt PGF}_{\infty}(R)$.
\end{Proposition}

\begin{proof}
Since $M$ is of type $FP_{\infty}$, it admits a projective
resolution
\[ \cdots \longrightarrow P_1 \longrightarrow P_0
   \longrightarrow M \longrightarrow 0 \]
consisting of finitely generated projective modules.
Since $M$ is Gorenstein flat and the class ${\tt GFlat}(R)$
is projectively resolving by \cite[Corollary 4.12]{SS}, all
kernels of that resolution are also Gorenstein flat. Splicing
that projective resolution of $M$ with the coresolution of
Lemma \ref{lem:seqs}(ii), we obtain an acyclic complex ${\bf P}$
\[ \cdots \longrightarrow P_1 \longrightarrow P_0
   \longrightarrow P_{-1} \longrightarrow P_{-2}
   \longrightarrow \cdots \]
consisting of finitely generated projective modules with
$M = Z_{-1}{\bf P}$ and all kernels $(Z_n{\bf P})_n$ Gorenstein
flat. It follows that ${\bf P}$ remains acyclic by applying the
functor $I \otimes_R -$ for any injective right $R$-module
$I$. Hence, all kernels of ${\bf P}$ are projectively coresolved
Gorenstein flat; in particular, $M \in {\tt PGF}(R)$.
\end{proof}

\begin{Corollary}\label{cor:GP=GF}
There are equalities
${\tt GProj}_{\infty}(R) = {\tt PGF}_{\infty}(R) =
 {\tt GFlat}_{\infty}(R)$.
\end{Corollary}

\begin{Corollary}\label{cor:GP-PGF-GF-dim}
If $M$ is a module of type $FP_{\infty}$, then
$\mathrm{Gpd}_RM = \mathrm{PGF\mbox{-}dim}_RM = \mathrm{Gfd}_RM$.
\end{Corollary}

\begin{proof}
Since the classes ${\tt GProj}(R)$, ${\tt PGF}(R)$ and
${\tt GFlat}(R)$ are projectively resolving and closed
under direct summands (cf.\ \cite[Theorem 2.5]{H},
\cite[Theorem 4.9 and Corollary 4.12]{SS}), the finiteness
of the relevant dimensions can be detected by the syzygy
modules in any projective resolution of $M$. Being of type
$FP_{\infty}$, the module $M$ admits a projective resolution
by finitely generated projective modules. We consider the
kernels $(\Omega_nM)_n$ of that particular projective resolution
and note that these are all modules of type $FP_{\infty}$.
Then, Corollary \ref{cor:GP=GF} implies that for all $n$ each one of the
following three conditions implies the other two:

(i) $\Omega^nM \in {\tt GProj}(R)$,

(ii) $\Omega^nM \in {\tt PGF}(R)$ and

(iii) $\Omega^nM \in {\tt GFlat}(R)$.
\newline
The equalities
$\mathrm{Gpd}_RM = \mathrm{PGF\mbox{-}dim}_RM = \mathrm{Gfd}_RM$ follow
readily from this.
\end{proof}

\vspace{0.1in}
\noindent
{\bf Acknowledgements.}
The authors are grateful to Peter H. Kropholler 
and Olympia Talelli for their interest in this work.
\vspace{0.1in}

\noindent
{\bf Funding.}
I.\ Emmanouil and K.\ Golfis were supported by the Hellenic Foundation
for Research and Innovation (H.F.R.I.) under the ``3rd Call for H.F.R.I.\
Research Projects to Support Faculty Members and Researchers'', project
number 24921. W.\ Ren was supported by the Natural Science Foundation of
Chongqing, China (No.\ CSTB2025NSCQ-GPX1014).

\bigskip

{\footnotesize \noindent Ioannis Emmanouil and Konstantinos Golfis,\\
Department of Mathematics, University of Athens, Athens 15784, Greece \\
E-mail: {\tt emmanoui$\symbol{64}$math.uoa.gr, golfisk$\symbol{64}$math.uoa.gr}}

\medskip

{\footnotesize \noindent Wei Ren,\\
 School of Mathematical Sciences, Chongqing Normal University, Chongqing 401331, PR China\\
 E-mail: {\tt wren$\symbol{64}$cqnu.edu.cn}}


\vspace{0.1in}

\begin{thebibliography}{99}

\bibitem{ABHS}
Asadollahi J., Bahlekeh A., Hajizamani A.,
Salarian S.: \emph{On certain homological invariants of groups},
J.\ Algebra {\bf 335} (2011), 18--35


\bibitem{BDT}
Bahlekeh, A., Dembegioti, F., Talelli, O.:
\emph{Gorenstein dimension and proper actions}, Bull.\ London Math.\
Soc.\ {\bf 41} (2009), 859--871

\bibitem{BS}
Baumslag, G., Solitar, D.: \emph{Some two-generator one-relator
non-Hopfian groups}, Bull.\ Amer.\ Math.\ Soc.\ {\bf 68}(3) (1962), 199--201


\bibitem{Bie}
Bieri, R.: Homological Dimension of Disrecte Groups,
2nd ed., Queen Mary College Mathmatical Notes, 1981

\bibitem{BK}
Bridson, M.R., Kropholler,  P.H.: \emph{Dimension of elementary amenable groups},
J./ Reine Angew./ Math./ {\bf 699} (2015), 217-–243

\bibitem{Bro}
Brown, K.S.: \emph{Cohomology of Groups}, Graduate
Texts in Mathematics 87, Springer, Berlin-Heidelberg-New
York, 1982

\bibitem{Buc}
Buchweitz, R.-O.:
\emph{Maximal Cohen--Macaulay Modules and Tate Cohomology},
Mathematical Surveys and Monographs, Vol. 262,
American Mathematical Society, Providence, RI, 2021

\bibitem{Chr}
Christensen, L.W.: \emph{Gorenstein Dimensions}, Lecture Notes in
Mathematics vol.1747, Berlin: Springer-Verlag, 2000

\bibitem{CK}
Cornick, J., Kropholler, P.H.: \emph{On complete
resolutions}, Topology Appl.\ {\bf 78} (1997), 235--250

\bibitem{DE}
Dalezios, G., Emmanouil, I.: \emph{Homological dimension
based on a class of Gorenstein flat modules},
C.R.\ Math.\ Acad.\ Sci.\ Paris {\bf 361} (2023), 1429--1448

\bibitem{ET18}
Emmanouil, I., Talelli, O.: \emph{Gorenstein dimension
and group cohomology with group ring coefficients}, J.\ London
Math.\ Soc.\ {\bf 97} (2018), 306--324

\bibitem{ET25}
Emmanouil, I., Talelli, O.: \emph{Characteristic modules
and  Gorenstein (co-)homological dimension of groups}, J.\ Pure Appl.\
Algebra {\bf 229} (2025), 107830

\bibitem{EJ} Enochs, E.E., Jenda, O.M.G.: \emph{Relative Homological
Algebra}, De Gruyter Expositions in Mathematics no. 30, New York:
Walter De Gruyter, 2000

\bibitem{FN}
Flores, R.J., Nucinkis, B.E.A.: \emph{On Bredon homology of
elementary amenable groups}, Proc.\ Amer.\ Math.\ Soc.\
{\bf 135}(1), (2007), 5--11

\bibitem{GG}
Gedrich, T.V., Gruenberg, K.W.: \emph{Complete
cohomological functors on groups},
Topology Appl.\ {\bf 25} (1987), 203--223

\bibitem{Hap}
Happel, D.:
On Gorenstein algebras,
in \emph{Representation Theory of Finite Groups and Finite-Dimensional Algebras},
Progr.\ Math.\ {\bf 95}, Birkh\"auser, Basel, 1991, pp.~389--404

\bibitem{HL}
Hillman, J.A., Linnell, P.A.:
\emph{Elementary amenable groups of finite Hirsch length are locally-finite by virtually-solvable},
J.\ Aust.\ Math.\ Soc.\ Series A {\bf 52} (1992), 237--241

\bibitem{H}
Holm, H.: \emph{Gorenstein homological dimensions}, J.\ Pure
Appl.\ Algebra {\bf 189} (2004), 167--193

\bibitem{KS}
Kaperonis I. Stergiopoulou D.D.: \emph{Finiteness criteria for
Gorenstein homological dimension and some invariants of groups},
Comm.\ Algebra, {\bf 53}(9) (2025), 3643--3661

\bibitem{K86}
Kropholler, P.H.: \emph{Cohomological
dimension of soluble groups}, J.\ Pure Appl.\ Algebra
{\bf 43} (1986), 281--287

\bibitem{K93}
Kropholler, P.H.: \emph{On groups of type
$\mathrm{FP}_{\infty}$}, J.\ Pure Appl.\ Algebra {\bf 90} (1993), 55--67

\bibitem{K-MP}
Kropholler, P.H., Mart\'{\i}nez-P\'{e}rez, C.:
\emph{Homological dimension of elementary amenable groups},
J.\ Reine Angew.\ Math.\ {\bf 766} (2020), 45--60

\bibitem{KM}
Kropholler, P.H., Mislin, G.: \emph{Groups acting on
finite dimensional spaces with finite stabilizers},
Comment.\ Math.\ Helv.\ {\bf 73} (1998), 122–136

\bibitem{L}
Lam, T.Y.: \emph{Lectures on Modules and Rings},
Graduate Texts in Mathematics {\bf 189}, Springer,
Berlin-Heidelberg-New York, 1998

\bibitem{Mal}
Mal'cev, A.I.: \emph{On certain classes of infinite solvable groups},
Amer.\ Math.\ Soc.\ Transl.\ {\bf 2}(2) (1956), 1--21

\bibitem{MP-N10}
Mart\'{\i}nez-P\'{e}rez, C., Nucinkis, B.E.A.:
\emph{Virtually soluble groups of type $FP_\infty$},
Comment.\ Math.\ Helv.\ {\bf 85} (2010), 135-–150

\bibitem{M}
Mislin, G.: \emph{On the classifying space for proper actions},
In: Aguad\'{e}, J., Broto, C., Casacuberta, C. (eds) Cohomological Methods
in Homotopy Theory, Progress in Mathematics, Birkh\"{a}user, Basel, vol 196, 2001

\bibitem{R}
Ren, W.: \emph{Gorenstein cohomological
dimension and stable categories for groups}, Comm.\ Algebra
{\bf 53}(5) (2025), 1866--1882

\bibitem{RY}
Ren, W., Yang, G.: \emph{Gorenstein homological
dimension and some invariants of groups}, J.\ Commut.\ Algebra
{\bf 17} (2025), 45–-62

\bibitem{SS}
\v{S}aroch, J., \v{S}t$\!$'$\!$ov\'{i}\v{c}ek, J.:
\emph{Singular compactness and definability for $\Sigma$-cotorsion
and Gorenstein modules}, Selecta Math.\ (N.S.) {\bf 26}, 2020,
Paper No.\ 23

\bibitem{Sim}
Simson, D.:
\emph{A remark on projective dimension of flat modules},
Math.\ Ann.\ {\bf 209} (1974), 181--182

\bibitem{Stal}
Stallings, J.R.: \emph{On torsion-free groups with infinitely many ends},
Ann.\ Math. {\bf 88}(2) (1968), 312--334

\bibitem{Stam}
Stammbach, U.: \emph{On the weak homological dimension
of the group algebra of solvable groups},  J.\ London
Math.\ Soc.\ {\bf 2} (1970), 567--570

\bibitem{Ster}
Stergiopoulou D.D.: \emph{Gorenstein homological dimension of group extensions},
arXiv:2608.08747

\bibitem{Swan}
Swan, R.G.: \emph{Groups of cohomological dimension one},
J.\ Algebra {\bf 12} (1969), 585--610

\bibitem{T14} Talelli, O.: \emph{On the Gorenstein and cohomological
dimension of groups},
Proc.\ Amer.\ Math.\ Soc.\ {\bf 142} (2014), 1175–-1180

\end{thebibliography}
\end{document}